\documentclass[12pt,reqno,twoside]{amsart}

\usepackage[T1]{fontenc}
\usepackage[utf8]{inputenc}
\usepackage{lmodern}
\usepackage{microtype}
\usepackage[a4paper,inner=1in,outer=1in,top=1.2in,bottom=1.2in]{geometry}

\usepackage{setspace}
\usepackage{amsmath,amssymb,amsthm,mathtools}
\usepackage{bm}
\usepackage{mathrsfs}
\usepackage{dsfont}

\usepackage{graphicx,tikz,caption,subcaption}
\usepackage{enumitem}

\usepackage[colorlinks=true,linkcolor=blue!50!black,citecolor=green!40!black,urlcolor=magenta!60!black]{hyperref}
\usepackage[nameinlink,capitalize]{cleveref}

\usepackage{bm}
\usepackage{amsbsy}

\numberwithin{equation}{section}

\allowdisplaybreaks
\theoremstyle{plain}
\newtheorem{Th}{Theorem}[section]
\newtheorem{Lem}[Th]{Lemma}

\newtheorem{Cor}[Th]{Corollary}

\theoremstyle{definition}
\newtheorem{Def}[Th]{Definition}
\newtheorem{Rem}[Th]{Remark}

\DeclareMathOperator{\supp}{supp}

\begin{document}
\title[A Hardy--Rellich Biharmonic Heat Equation]{Double Criticality for a Hardy--Rellich Biharmonic Heat Equation in an Exterior Domain}

\author[H. Alhatlani]{Hadeel Alhatlani}
\address{(H. Alhatlani) Department of Mathematics, College of Science, King Saud University, Riyadh 11451, Saudi Arabia}
\email{444203494@student.ksu.edu.sa}

\author[M. Jleli]{Mohamed Jleli}
\address{(M. Jleli) Department of Mathematics, College of Science, King Saud University, Riyadh 11451, Saudi Arabia}
\email{jlelimohamed@gmail.com}

\author[B. Samet]{Bessem Samet}
\address{(B. Samet) Department of Mathematics, College of Science, King Saud University, Riyadh 11451, Saudi Arabia}
\email{bsamet@ksu.edu.sa}

\renewcommand{\subjclassname}{\textbf{2020 Mathematics Subject Classification}}

\begin{abstract} 
We study the existence and nonexistence of weak solutions to an inhomogeneous semilinear biharmonic heat equation in an exterior domain, involving a singular Hardy--Rellich potential, a weighted nonlinearity $|x|^{\sigma}|u|^{p}$, and a positive source term $f(x)$. We identify two distinct critical regimes governing the behavior of solutions. More precisely, we first determine a Fujita-type critical exponent that separates nonexistence from existence. We then show that, in the supercritical range, a second critical exponent arises in terms of the decay exponent of the source, in the sense of Lee and Ni. Our results extend the recent work \cite{Tobakhanov} by considering a singular Hardy--Rellich potential and a weighted nonlinearity, leading to a different critical behavior.
\end{abstract}

\maketitle

\tableofcontents

\section{Introduction and main results}

For $N \geq 5$, consider the exterior domain
\[
\Omega_e = \mathbb{R}^N \setminus \overline{B_1},
\]
where $B_1$ denotes the open unit ball in $\mathbb{R}^N$. In this paper, we are concerned with the semilinear biharmonic heat equation
\begin{equation}\label{P}
u_t + \Delta^2 u - \frac{\mu}{|x|^4} u = |x|^{\sigma} |u|^p + f(x),
\qquad t > 0,\ x \in \Omega_e,
\end{equation}
subject to the Navier boundary conditions
\begin{equation}\label{BC}
u = \Delta u = 0, \qquad t > 0,\ x \in \partial B_1.
\end{equation}
Here, $p > 1$, $\sigma > -4$, and $f > 0$ is a function in $L^1_{\mathrm{loc}}(\overline{\Omega_e})$. The parameter $\mu$ satisfies
\begin{equation}\label{mu-as}
0 < \mu \leq \mu^* = \left( \frac{N(N-4)}{4} \right)^2.
\end{equation}
The value $\mu^*$ is the best constant in the Hardy--Rellich inequality; see \cite{Tertikas-2007}.

Problem \eqref{P}--\eqref{BC} describes a fourth-order diffusion process in the exterior of the unit ball, involving a singular Hardy--Rellich potential, a nonlinear reaction term weighted by $|x|^\sigma$, and an external source term.

Our main goal is to investigate the critical behavior associated with the existence and nonexistence of weak solutions to \eqref{P}--\eqref{BC}.

Weak solutions to \eqref{P}--\eqref{BC} are understood in the following sense. We set
\[
D = (0,\infty) \times \overline{\Omega_e}, \qquad \Sigma = (0,\infty) \times \partial B_1.
\]
We consider the class of test functions
\[
\Phi = \left\{ \varphi \in C_c^{1,4}(D) : \varphi \geq 0,\ \varphi = \Delta \varphi = 0 \ \text{on } \Sigma \right\}.
\]
Here, $C_c^{1,4}(D)$ denotes the space of functions $\varphi(t,x)$ that are of class $C^1$ in $t$, $C^4$ in $x$, and have compact support in $D$.

\begin{Def}\label{def-ws}
A function $u \in L^p_{\mathrm{loc}}(D)$ is said to be  a weak solution to \eqref{P}--\eqref{BC} if
\[
\int_D |x|^{\sigma} |u|^p \varphi \, dx\, dt
+ \int_D f(x)\, \varphi \, dx\, dt
=
- \int_D u\, \varphi_t \, dx\, dt
+ \int_D u \left( \Delta^2 \varphi - \frac{\mu}{|x|^4} \varphi \right) dx\, dt
\]
for all $\varphi \in \Phi$.
\end{Def}

The above weak formulation is obtained by multiplying \eqref{P} by $\varphi \in \Phi$, integrating over $D$, and integrating by parts, using the Navier boundary conditions.

Before stating our main results, we  recall some works related to the present study. It is well known that semilinear problems possess a critical behavior, first identified in the seminal work of Fujita \cite{Fujita}. In his pioneering work, Fujita considered the Cauchy problem for the semilinear heat equation
\[
u_t - \Delta u = u^p, \qquad t>0,\ x \in \mathbb{R}^N,
\]
supplemented with nonnegative initial data
\[
u(0,x) = u_0(x), \qquad x \in \mathbb{R}^N.
\]
He proved that:
\begin{itemize}
\item[(i)] If $1 < p < 1 + \frac{2}{N}$ and $u_0 \ge 0$, $u_0 \not\equiv 0$, then the problem admits no global positive solution.
\item[(ii)] If $p > 1 + \frac{2}{N}$ and $u_0$ is sufficiently small (for instance, dominated by a Gaussian), then the problem admits global positive solutions.
\end{itemize}
This result shows that 
\[
p_F = 1 + \frac{2}{N}
\] 
is a critical exponent, known as the Fujita critical exponent, which separates existence from nonexistence of global solutions. Moreover, the critical value $p = p_F$ lies in case (i) (see Hayakawa~\cite{Hayakawa} for $N = 1,2$, and Sugitani~\cite{Sugitani} and Kobayashi et al.~\cite{Kobayashi} for $N \ge 3$). 

Subsequently, Lee and Ni~\cite{Lee} studied the behavior of solutions to the above problem in the supercritical range $p > p_F$, assuming that $u_0 \ge 0$ and
\[
u_0(x) \sim |x|^{-\eta} \quad \text{as } |x| \to \infty.
\]
They established that solutions cannot exist globally when $\eta < \frac{2}{p-1}$, while global solutions do exist provided that $\eta \ge \frac{2}{p-1}$ and the initial data is sufficiently small. The value
\[
\eta_{\mathrm{crit}} = \frac{2}{p-1}
\]
is referred to as the second critical exponent.

Consider now the semilinear biharmonic heat equation
\[
u_t + \Delta^2 u = |u|^p, \qquad t>0,\ x \in \mathbb{R}^N.
\]
The results of Egorov et al. \cite{Egorov} show that, if $1 < p \leq 1 + \frac{4}{N}$, then for any initial datum with positive integral, the corresponding solution does not exist globally. Moreover, if $p > 1 + \frac{4}{N}$, there exist global solutions for suitable initial data; see Galaktionov and Pohozaev \cite{Galaktionov}. Further related results can be found in  \cite{Caristi,Gazzola,Philippin}. 

Fujita-type results have also been established in various domains other than $\mathbb{R}^N$. In particular, numerous results have been obtained in exterior domains. For instance, Bandle and Levine \cite{Bandle} studied the semilinear heat equation
\[
u_t - \Delta u = u^p\,\,(u\geq 0), \qquad t>0,\ x \in \mathbb{R}^N \setminus \Omega,
\]
under the Dirichlet boundary condition 
\begin{equation}\label{DRbc}
u(t,x)=0,\qquad  t>0,\ x \in \partial \Omega,
\end{equation}
where $\Omega$ is a bounded domain in $\mathbb{R}^N$. It was shown that the Fujita critical exponent remains unchanged in this setting. For the critical case $p = p_F$, we refer to Ikeda and Sobajima \cite{Ikeda} in the case $N = 2$, and to Suzuki \cite{Suzuki} for $N \geq 3$. In a subsequent work, Levine and Zhang \cite{Levine} studied the same problem under the homogeneous Neumann boundary condition
\[
\frac{\partial u}{\partial \nu} = 0,\qquad  t>0,\ x \in \partial \Omega.
\]
They showed that the Fujita critical exponent remains $p_F$ in this case as well. Further related results can be found in \cite{Borikhanov,JS,Rault,Sun,Zhang}.

Bandle et al. \cite{Bandle2} studied the corresponding inhomogeneous problem
\[
u_t - \Delta u = |u|^p + f(x), \qquad t>0,\ x \in \mathbb{R}^N \setminus \Omega,
\]
subject to the  Dirichlet boundary condition \eqref{DRbc}.   They showed that the presence of the inhomogeneous term leads to a different critical behavior, with critical exponent $1+\frac{2}{N-2}$ ($N\geq 3$), which is larger than the Fujita exponent.

Very recently, Tobakhanov and Torebek \cite{Tobakhanov} studied the special case of \eqref{P} corresponding to $\mu = 0$ and $\sigma = 0$. Namely, they considered the semilinear biharmonic heat equation
\begin{equation}\label{P-TT}
u_t + \Delta^2 u = |u|^p + f(x),
\qquad t > 0,\ x \in \Omega_e,
\end{equation} 
under different types of homogeneous boundary conditions. In particular, under the Navier boundary conditions \eqref{BC}, they obtained the following results (for $N\geq 5$):
\begin{itemize}
	\item[(i)] If $1 < p \leq 1 + \frac{4}{N-4}$, then no weak solution exists, provided that
\[
\int_{\Omega_e}  f(x)A(x)\,dx > 0
\]
for a certain weight $A$.
\item[(ii)] If $p > 1 + \frac{4}{N-4}$, then there exist stationary solutions for some $f > 0$.
\end{itemize}
Thus, 
\[
1 + \frac{4}{N-4}
\]
is the critical exponent separating nonexistence from existence.

In the same work \cite{Tobakhanov}, the authors also analyzed the critical behavior in the sense  by Lee and Ni~\cite{Lee}. Namely, for 
\[
p>1 + \frac{4}{N-4},
\]
they proved the following results:

\begin{itemize}
	\item[(i)] Let $f \geq 0$ be continuous and satisfy $f(x) \geq c |x|^{-\eta}$ for sufficiently large $|x|$. If
\[
\eta < \frac{4p}{p-1},
\]
then no weak solution exists.
\item[(ii)] If
\[
\frac{4p}{p-1} \leq \eta < N,
\]
then the problem admits stationary solutions for some $f \geq 0$ that is continuous and satisfies $f(x) \leq c |x|^{-\eta}$ for sufficiently large $|x|$.
\end{itemize}

The above results reveal a second critical behavior, characterized by the critical decay rate
\[
\eta_{\mathrm{crit}} = \frac{4p}{p-1}.
\]

Semilinear heat equations involving singular potentials, and in particular Hardy-type potentials, have also attracted considerable attention. For instance, Abdellaoui et al. \cite{Abdellaoui} considered semilinear heat equations of the form
\[
u_t - \Delta u - \frac{\lambda}{|x|^2}u = u^p + f(x), \qquad t > 0,\ x \in \Omega,
\]
under the Dirichlet boundary condition
\[
u(t,x) = 0, \qquad t > 0,\ x \in \partial \Omega,
\]
where $\Omega \subset \mathbb{R}^N$, $N \geq 3$, is a bounded regular domain containing the origin, $p > 1$, and $u_0, f \geq 0$ belong to suitable function classes. It was shown that if $\lambda > 0$, there exists a critical exponent $p_+(\lambda)$ such that, for $p \geq p_+(\lambda)$, no solution exists for any nontrivial initial datum. Further related results can be found in \cite{Abdellaoui16,Abdellaoui2,JS24}, as well as in the references therein.

The elliptic counterpart has also been extensively investigated, particularly for equations involving singular potentials of Hardy type; see, for example, \cite{Brezis,Chen,Chen2,D'Ambrosio} and the references therein.

In this paper, motivated by the recent work of Tobakhanov and Torebek~\cite{Tobakhanov}, we investigate the critical behavior associated with \eqref{P}--\eqref{BC}. We show that the problem possesses  a double critical structure: a first critical behavior governed by the exponent $p$, in the spirit of Fujita, and a second one governed by the decay rate of the source term, in the sense of Lee and Ni, arising from the presence of a singular Hardy--Rellich potential. To the best of our knowledge, this type of double critical behavior has not been previously established for biharmonic heat equations with a Hardy--Rellich potential in exterior domains.

We now state our main results. For $0 < \mu \le \mu^*$, we introduce the parameter $\mu_N$ defined by
\[
\mu_N = \frac{N-4}{2} + \sqrt{
\left(\frac{N-2}{2}\right)^2 + 1 - \sqrt{(N-2)^2 + \mu}
}.
\]

Our first main result reads as follows.

\begin{Th}\label{T1.2}
Let $N \ge 5$ and $\sigma > -4$.
\begin{itemize}
\item[\rm(i)] Assume that either
\begin{equation}\label{cd-nex}
0 < \mu \le \mu^*, \qquad 1 < p < 1 + \frac{\sigma + 4}{\mu_N},
\end{equation}
or
\begin{equation}\label{cd-nex-cr}
0 < \mu < \mu^*, \qquad p = 1 + \frac{\sigma + 4}{\mu_N}.
\end{equation}
Then, for every positive function $f \in L^1_{\mathrm{loc}}(\overline{\Omega_e})$, problem \eqref{P}--\eqref{BC} admits no weak solution.

\item[\rm(ii)] If $0 < \mu \le \mu^*$ and
\begin{equation}\label{cd-ex}
p > 1 + \frac{\sigma + 4}{\mu_N},
\end{equation}
then problem \eqref{P}--\eqref{BC} admits stationary solutions $u \in C^\infty(\overline{\Omega_e})$ for some positive function $f$.
\end{itemize}
\end{Th}

\begin{Rem}
Theorem~\ref{T1.2} shows that there exists a critical exponent
\[
p_{\mathrm{crit}}=p_{\mathrm{crit}}(\mu,\sigma) = 1 + \frac{\sigma + 4}{\mu_N},
\]
which governs the behavior of solutions. More precisely, this exponent separates the nonexistence and existence regimes: if $1 < p <p_{\mathrm{crit}}$, then no weak solution exists, whereas if $p > p_{\mathrm{crit}}$, stationary solutions do exist for suitable positive sources.
\end{Rem}

\begin{Rem}
The critical exponent $p_{\mathrm{crit}}$ depends explicitly on the parameter $\mu$ through $\mu_N$. In comparison with the case $\mu = 0$ and $\sigma = 0$ studied by Tobakhanov and Torebek~\cite{Tobakhanov}, where the critical exponent is
\[
1 + \frac{4}{N-4},
\]
we see that the presence of  the Hardy--Rellich potential  changes the form of the critical exponent. In particular, when $\sigma = 0$, one has
\[
p_{\mathrm{crit}}(\mu,0) \downarrow 1 + \frac{4}{N-4}
\qquad \text{as } \mu \to 0^+.
\]
Moreover, when $\mu=\mu^*$, one has
\[
p_{\mathrm{crit}}(\mu^*,\sigma)=1 + \frac{2(\sigma + 4)}{N-4},\qquad p_{\mathrm{crit}}(\mu^*,0)=1 + \frac{8}{N-4}.
\]
\end{Rem}

\begin{Rem}
The classification given in Theorem~\ref{T1.2} is complete in the case $0 < \mu < \mu^*$. However, when $\mu = \mu^*$, the critical case $p = p_{\mathrm{crit}}$ is not covered by our result. The difficulty is related to the degeneracy of the associated operator at the critical value $\mu = \mu^*$. The behavior of solutions in this situation remains an open question.
\end{Rem}

Consider now the associated fourth-order elliptic problem obtained from \eqref{P} by removing the time variable, namely,
\begin{equation}\label{E}
\begin{cases}
\displaystyle \Delta^2 u - \frac{\mu}{|x|^4} u = |x|^{\sigma} |u|^p + f(x), & x \in \Omega_e, \\[6pt]
u = \Delta u = 0, & x \in \partial B_1.
\end{cases}
\end{equation}

From Theorem~\ref{T1.2}, we deduce the following result for problem \eqref{E}.

\begin{Cor}
Let $N \ge 5$ and $\sigma > -4$.
\begin{itemize}
\item[\rm(i)] Assume that either \eqref{cd-nex} or \eqref{cd-nex-cr} holds. Then problem \eqref{E} admits no weak solution for every positive function $f \in L^1_{\mathrm{loc}}(\overline{\Omega_e})$.

\item[\rm(ii)] If $0 < \mu \le \mu^*$ and \eqref{cd-ex} holds, then problem \eqref{E} admits smooth solutions $u \in C^\infty(\overline{\Omega_e})$ for some positive function $f$.
\end{itemize}
\end{Cor}

We next refine our analysis by taking into account the decay rate of the source term. For
\[
p > 1 + \frac{\sigma + 4}{\mu_N}, \qquad \eta < \mu_N + 4,
\]
we introduce the classes
\[
\mathcal{F}^+_{\eta}
=
\left\{
f \in C(\overline{\Omega_e}) :\ 
f > 0,\ 
f(x) \ge c |x|^{-\eta} \text{ for sufficiently large } |x|
\right\}
\]
and
\[
\mathcal{F}^-_{\eta}
=
\left\{
f \in C(\overline{\Omega_e}) :\ 
f > 0,\ 
f(x) \le c |x|^{-\eta} \text{ for sufficiently large } |x|
\right\},
\]
where $c > 0$ is a constant.

Our second main result is stated as follows.
\begin{Th}\label{T1.3}
Let $N \ge 5$, $\sigma > -4$, $0 < \mu \le \mu^*$, and assume that \eqref{cd-ex} holds.
\begin{itemize}
\item[\rm(i)] If
\begin{equation}\label{cd-nex-eta}
\eta < 4 + \frac{\sigma + 4}{p-1},
\end{equation}
then, for every $f \in \mathcal{F}^+_{\eta}$, problem \eqref{P}--\eqref{BC} admits no weak solution.

\item[\rm(ii)] If
\begin{equation}\label{cd-ex-eta}
4 + \frac{\sigma + 4}{p-1} \le \eta < \mu_N + 4,
\end{equation}
then problem \eqref{P}--\eqref{BC} admits stationary solutions $u \in C^\infty(\overline{\Omega_e})$ for some $f \in \mathcal{F}^-_{\eta}$.
\end{itemize}
\end{Th}

\begin{Rem}
Theorem~\ref{T1.3} establishes a second critical behavior in terms of the parameter $\eta$, in the sense of Lee and Ni. More precisely, the value
\[
\eta_{\mathrm{crit}} = 4 + \frac{\sigma + 4}{p-1}
\]
separates the nonexistence and existence regimes. In particular, the critical case $\eta = \eta_{\mathrm{crit}}$ belongs to the existence regime. This highlights a double critical structure, where the exponent $p$ and the decay rate of the source term jointly determine the behavior of solutions.
\end{Rem}

\begin{Rem}
In the particular case $\sigma = 0$ and as $\mu \to 0^+$, we obtain
\[
\eta_{\mathrm{crit}} = \frac{4p}{p-1},
\]
with the additional condition $\eta < N$. This coincides with the critical exponent obtained in~\cite{Tobakhanov}.
\end{Rem}

From Theorem~\ref{T1.3}, we deduce the following result for problem \eqref{E}.

\begin{Cor}
Let $N \ge 5$, $\sigma > -4$, $0 < \mu \le \mu^*$, and assume that \eqref{cd-ex} holds.
\begin{itemize}
\item[\rm(i)] If \eqref{cd-nex-eta} holds, then problem \eqref{E} admits no weak solution for every $f \in \mathcal{F}^+_{\eta}$.

\item[\rm(ii)] If \eqref{cd-ex-eta} holds, then problem \eqref{E} admits smooth solutions $u \in C^\infty(\overline{\Omega_e})$ for some $f \in \mathcal{F}^-_{\eta}$.
\end{itemize}
\end{Cor}

The nonexistence results rely on the construction of suitable functions in $\Phi$. We construct functions of the first kind for $1 < p < p_{\mathrm{crit}}$, and functions of the second kind for the critical case $p = p_{\mathrm{crit}}$. The presence of the Hardy--Rellich potential makes this construction more delicate than in~\cite{Tobakhanov}. In particular, a detailed analysis of the associated biharmonic equation
\[
\Delta^2 H - \frac{\mu}{|x|^4} H = 0 \quad \text{in } \Omega_e,
\]
subject to Navier boundary conditions, is required in both regimes $0 < \mu < \mu^*$ and $\mu = \mu^*$.

The nonexistence results rely on the construction of suitable functions in $\Phi$. We construct functions of the first kind for $1 < p < p_{\mathrm{crit}}$, and functions of the second kind for the critical case $p = p_{\mathrm{crit}}$. In contrast with the case $\mu = 0$ considered in~\cite{Tobakhanov}, the presence of the Hardy--Rellich potential fundamentally alters the structure of the problem and prevents a direct application of their approach. In particular, a detailed analysis of the associated biharmonic equation
\[
\Delta^2 H - \frac{\mu}{|x|^4} H = 0 \quad \text{in } \Omega_e,
\]
subject to Navier boundary conditions, is required in both regimes $0 < \mu < \mu^*$ and $\mu = \mu^*$.

For the existence part, we construct explicit stationary solutions by analyzing separately the cases $0 < \mu < \mu^*$ and $\mu = \mu^*$, where the structure of the associated operator differs significantly.

The structure of the paper is as follows. Section~\ref{sec2} is devoted to the construction of admissible functions in $\Phi$. We first study the associated homogeneous Hardy--Rellich-type equation and construct a suitable barrier function $H$ in the regimes $0 < \mu < \mu^*$ and $\mu = \mu^*$. We then build test functions of the first and second kinds, which are used in the proof of the nonexistence results. Section~\ref{sec3} is devoted to the derivation of auxiliary estimates required in the proof of the nonexistence results. Section~\ref{sec4} contains the proofs of Theorems~\ref{T1.2} and~\ref{T1.3}.

\section{Construction of admissible functions}\label{sec2}

In this section, we construct admissible test functions $\varphi \in \Phi$ for \eqref{P}--\eqref{BC}, which play a key role in the derivation of the nonexistence results.

\subsection{The homogeneous Hardy--Rellich--type equation}

For $\mu$ satisfying \eqref{mu-as}, we introduce the fourth-order differential operator
\[
\mathcal{L}_\mu = \Delta^2 - \frac{\mu}{|x|^4}.
\]
We then consider the homogeneous problem
\begin{equation}\label{HP}
\begin{cases}
\mathcal{L}_\mu H = 0, & x \in \Omega_e,\\
H = \Delta H = 0, & x \in \partial B_1,\\
H \geq 0, & x \in \Omega_e.
\end{cases}
\end{equation}

We work with the parameters
\[
a = \frac{N-4}{2}, \qquad
A(\mu) = (a+1)^2 + 1 - \sqrt{4(a+1)^2 + \mu}, \qquad
B(\mu) = (a+1)^2 + 1 + \sqrt{4(a+1)^2 + \mu}.
\]
We define the functions
\begin{equation}\label{Plambda}
P(\lambda) = \lambda(\lambda-2)(\lambda+N-2)(\lambda+N-4) - \mu,
\end{equation}
and
\begin{equation}\label{Mlambda}
M(\lambda) = \lambda(\lambda+N-2), \qquad \lambda \in \mathbb{R}.
\end{equation}

\subsubsection{The case $0 < \mu < \mu^*$}

Let
\[
\alpha_\pm = -a \pm \sqrt{A(\mu)}, \qquad
\beta_\pm = -a \pm \sqrt{B(\mu)}.
\]
Then $\alpha_\pm$ and $\beta_\pm$ are the four roots of the polynomial $P(\lambda)$. Moreover, one has
\[
\beta_- < \alpha_- < \alpha_+ < 0 < \beta_+.
\]
We set
\[
c_\mu = \frac{M(\beta_-) - M(\alpha_+)}{2(\alpha_- - \alpha_+)}.
\]

\begin{Lem}\label{L2.1}
If $0 < \mu < \mu^*$, then $c_\mu < 0$.
\end{Lem}

\begin{proof}
Since $\alpha_- < \alpha_+$, it suffices to show that $M(\beta_-) > M(\alpha_+)$. By the definition of $M$, we have
\[
M(\beta_-) - M(\alpha_+)
= (\beta_- - \alpha_+)(\beta_- + \alpha_+ + N - 2).
\]
Since $\beta_- < \alpha_+$, the first factor is negative. It therefore suffices to prove that
\begin{equation}\label{claim}
\beta_- + \alpha_+ + N - 2 < 0.
\end{equation}
By the definitions of $\beta_-$ and $\alpha_+$, we have
\[
\beta_- + \alpha_+ + N - 2 = 2 + \sqrt{A(\mu)} - \sqrt{B(\mu)}.
\]
Using the inequalities $A(\mu) < a^2$ and $B(\mu) > (a+2)^2$, we obtain \eqref{claim}. The proof is complete.
\end{proof}

We now introduce the function
\begin{equation}\label{H-f}
H(x) = c_\mu |x|^{\alpha_-} + (1 - c_\mu) |x|^{\alpha_+} - |x|^{\beta_-},
\qquad x \in \overline{\Omega_e}.
\end{equation}

\begin{Lem}\label{L2.2}
Let $0 < \mu < \mu^*$. The function $H$ defined by \eqref{H-f} satisfies \eqref{HP}.
\end{Lem}

\begin{proof}
For $\lambda \in \mathbb{R}$, let $v_\lambda(x) = |x|^\lambda$, $x \in \Omega_e$. A direct computation shows that $\mathcal{L}_\mu v_\lambda = 0$ if and only if $P(\lambda) = 0$. Since $\alpha_\pm$ and $\beta_\pm$ are roots of $P$, it follows that
\[
\mathcal{L}_\mu(|x|^{\alpha_-})
=
\mathcal{L}_\mu(|x|^{\alpha_+})
=
\mathcal{L}_\mu(|x|^{\beta_-})
= 0
\qquad \text{in } \Omega_e.
\]
By linearity of $\mathcal{L}_\mu$ and the definition of $H$, we conclude that $\mathcal{L}_\mu H = 0$ in $\Omega_e$.

By definition of $H$, we have $H = 0$ on $\partial B_1$. Moreover,
\[
\Delta H(x)
=
c_\mu M(\alpha_-) |x|^{\alpha_- - 2}
+ (1 - c_\mu) M(\alpha_+) |x|^{\alpha_+ - 2}
- M(\beta_-) |x|^{\beta_- - 2}.
\]
Therefore, on $\partial B_1$,
\[
\Delta H
=
c_\mu M(\alpha_-)
+ (1 - c_\mu) M(\alpha_+)
- M(\beta_-)
= 0,
\]
by the definition of $c_\mu$.

It remains to show that $H \geq 0$ in $\Omega_e$. Let
\[
h(r) = H(x) = c_\mu r^{\alpha_-} + (1 - c_\mu) r^{\alpha_+} - r^{\beta_-}, \qquad r = |x| \geq 1.
\]
We introduce the function
\[
F(t) = e^{-\beta_- t} h(e^t), \qquad t \geq 0,
\]
that is,
\[
F(t) = c_\mu e^{(\alpha_- - \beta_-)t} + (1 - c_\mu) e^{(\alpha_+ - \beta_-)t} - 1.
\]
We set
\[
\gamma = \alpha_- - \beta_- > 0, \qquad \delta = \alpha_+ - \beta_- > 0.
\]
Then
\[
F(t) = c_\mu e^{\gamma t} + (1 - c_\mu) e^{\delta t} - 1.
\]

A direct computation gives
\[
F(0) = 0,
\]
and
\[
F'(t) = c_\mu \gamma e^{\gamma t} + (1 - c_\mu)\delta e^{\delta t}.
\]
Using the definition of $c_\mu$, one checks that
\[
F'(0)
=
\frac{(\alpha_+ - \beta_-)(\alpha_- - \beta_-)}{2} > 0.
\]
Moreover,
\[
F''(t)
=
e^{\gamma t} \Big( c_\mu \gamma^2 + (1 - c_\mu)\delta^2 e^{(\delta - \gamma)t} \Big).
\]
Since $1 - c_\mu > 0$ (by Lemma~\ref{L2.1}) and $\delta - \gamma = \alpha_+ - \alpha_- > 0$, the function
\[
t \mapsto c_\mu \gamma^2 + (1 - c_\mu)\delta^2 e^{(\delta - \gamma)t}
\]
is increasing on $[0,\infty)$. Therefore,
\[
F''(t) \ge e^{\gamma t} F''(0), \qquad t \ge 0.
\]
Moreover, a direct computation yields
\[
F''(0) = \frac{\gamma \delta (\gamma + \delta - 2)}{2} > 0.
\]
Hence,
\[
F''(t) > 0 \qquad \text{for all } t \ge 0.
\]
It follows that $F$ is strictly convex on $[0,\infty)$. Since $F(0) = 0$ and $F'(0) > 0$, we conclude that
\[
F(t) \ge 0 \quad \text{for all } t \ge 0.
\]

Hence,
\[
h(r) \ge 0 \quad \text{for all } r \ge 1,
\]
that is,
\[
H(x) \ge 0 \quad \text{for all } x \in \Omega_e.
\]
The proof is complete.
\end{proof}

\subsubsection{The case $\mu = \mu^*$}

In this case, we have
\[
A(\mu^*) = 0.
\]
Hence, the two roots $\alpha_-$ and $\alpha_+$ coincide. More precisely, setting
\[
\alpha = -a = -\frac{N-4}{2},
\]
we obtain
\[
\alpha_- = \alpha_+ = \alpha.
\]
Moreover,
\[
\beta_- < \alpha < 0,
\]
where
\[
\beta_- = -a - \sqrt{B(\mu^*)}.
\]

Let
\[
d_{\mu^*} = \frac{M(\beta_-) - M(\alpha)}{2}.
\]
We define
\begin{equation}\label{H-critical}
H(x) = |x|^\alpha + d_{\mu^*} |x|^\alpha \ln |x| - |x|^{\beta_-},
\qquad x \in \overline{\Omega_e}.
\end{equation}

\begin{Lem}\label{L2.3}
Let $\mu = \mu^*$. The function $H$ defined by \eqref{H-critical} satisfies \eqref{HP}.
\end{Lem}

\begin{proof}
Since $\alpha$ is a double root of $P$, we have
\[
\mathcal{L}_{\mu^*}(|x|^\alpha) = 0
\qquad \text{in } \Omega_e.
\]
Moreover, a second linearly independent solution associated with this double root is given by
\[
|x|^\alpha \ln |x|.
\]
Since $\beta_-$ is also a root of $P$, we have
\[
\mathcal{L}_{\mu^*}(|x|^{\beta_-}) = 0
\qquad \text{in } \Omega_e.
\]
By linearity, it follows that $\mathcal{L}_{\mu^*} H = 0$ in $\Omega_e$.

By definition of $H$, we have $H = 0$ on $\partial B_1$. Using
\[
\Delta(|x|^\lambda) = M(\lambda) |x|^{\lambda - 2},
\]
we obtain
\[
\Delta\big(|x|^\alpha \ln |x|\big)
=
M'(\alpha) |x|^{\alpha - 2}
+ M(\alpha) |x|^{\alpha - 2} \ln |x|.
\]
Therefore,
\[
\Delta H(x)
=
M(\alpha) |x|^{\alpha - 2}
+ d_{\mu^*} \left( M'(\alpha) |x|^{\alpha - 2}
+ M(\alpha) |x|^{\alpha - 2} \ln |x| \right)
- M(\beta_-) |x|^{\beta_- - 2}.
\]
In particular, for $x \in \partial B_1$, we have
\[
\Delta H(x)
=
M(\alpha) + d_{\mu^*} M'(\alpha) - M(\beta_-)
= 0.
\]

It remains to show that $H \geq 0$ in $\Omega_e$. Let $r = |x| \geq 1$ and write
\[
H(x) = h(r) = r^\alpha + d_{\mu^*} r^\alpha \ln r - r^{\beta_-}.
\]
We set $r = e^t$, $t \geq 0$, and define
\[
F(t) = e^{-\beta_- t} h(e^t).
\]
Then
\[
F(t) = e^{(\alpha - \beta_-)t}(1 + d_{\mu^*} t) - 1.
\]
Moreover,
\[
F(0) = 0,
\]
and
\[
F'(t)
=
e^{(\alpha - \beta_-)t}
\left[
(\alpha - \beta_-)(1 + d_{\mu^*} t) + d_{\mu^*}
\right].
\]
Since
\[
M(\beta_-) - M(\alpha)
=
(\beta_- - \alpha)(\beta_- + \alpha + N - 2),
\]
and
\[
\beta_- - \alpha < 0,
\qquad
\beta_- + \alpha + N - 2 < 0,
\]
we obtain
\[
d_{\mu^*} > 0.
\]
It follows that $F'(t) > 0$ for all $t \geq 0$. Hence $F$ is increasing on $[0,\infty)$. Since $F(0) = 0$, we conclude that
\[
F(t) \geq 0
\qquad \text{for all } t \geq 0.
\]
Consequently,
\[
h(r) \geq 0
\qquad \text{for all } r \geq 1,
\]
which yields $H(x) \geq 0$ for all $x \in \Omega_e$. The proof is complete.
\end{proof}

\subsection{Construction of test functions of the first kind}

For $\ell > 0$, the notation $\ell \gg 1$ means that $\ell$ is sufficiently large.

Let $\varrho \in C^\infty([0,\infty))$ be such that $\varrho \geq 0$, $\varrho \not\equiv 0$, and $\supp \varrho \subset (0,1)$. For $\ell, T \gg 1$, we define
\[
\varrho_T(t) = \varrho^\ell\!\left(\frac{t}{T}\right), \qquad t > 0.
\]

We also consider $\xi \in C^\infty([0,\infty))$ such that $0 \le \xi \le 1$, $\xi(s) = 1$ for $0 \le s \le \tfrac{1}{2}$, and $\xi(s) = 0$ for $s \ge 1$. For $\ell, R \gg 1$, we define
\[
\xi_R(x) = \xi^\ell\!\left(\frac{|x|}{R}\right), \qquad 
H_R(x) = H(x)\, \xi_R(x), \qquad x \in \overline{\Omega_e},
\]
where $H$ is given by \eqref{H-f} if $0 < \mu < \mu^*$, and by \eqref{H-critical} if $\mu = \mu^*$.

For $\ell, T, R \gg 1$, we define
\begin{equation}\label{test1}
\varphi(t,x) = \varrho_T(t)\, H_R(x), \qquad (t,x) \in D.
\end{equation}

\begin{Lem}\label{lem:test1}
For $\ell, T, R \gg 1$, the function $\varphi$ defined by \eqref{test1} belongs to $\Phi$.
\end{Lem}

\begin{proof}
By construction, $\varrho_T \in C_c^1(0,\infty)$ and $\xi_R \in C_c^4(\overline{\Omega_e})$. Since $H$ is smooth in $\overline{\Omega_e}$, we have
\[
\varphi \in C_c^{1,4}(D).
\]
Moreover, by Lemmas \ref{L2.2} and \ref{L2.3}, we have $H \geq 0$ in $\Omega_e$. Since $\varrho_T \geq 0$ and $\xi_R \geq 0$, it follows that
\[
\varphi \geq 0 \qquad \text{in } D.
\]
Furthermore,  we have $H = \Delta H = 0$ on $\partial B_1$. Since $\xi_R = 1$ in a neighborhood of $\partial B_1$ for $R \gg 1$, we obtain
\[
\varphi = \Delta \varphi = 0
\qquad \text{on } \Sigma.
\]
Therefore, $\varphi \in \Phi$.
\end{proof}

\subsection{Construction of test functions of the second kind}

We choose $\zeta \in C^\infty(\mathbb{R})$ such that $0 \le \zeta \le 1$, $\zeta(s) = 1$ for $s \le 0$, and $\zeta(s) = 0$ for $s \ge 1$. For $\ell, R \gg 1$, we define
\[
\zeta_R(x) = \zeta^\ell\!\left(\frac{2\ln |x|}{\ln R} - 1\right), \qquad 
\overline{H}_R(x) = H(x)\, \zeta_R(x), \qquad x \in \overline{\Omega_e},
\]
where $H$ is given by \eqref{H-f} if $0 < \mu < \mu^*$, and by \eqref{H-critical} if $\mu = \mu^*$.

For $\ell, T, R \gg 1$, we define
\begin{equation}\label{test2}
\varphi(t,x) = \varrho_T(t)\, \overline{H}_R(x), \qquad (t,x) \in D.
\end{equation}

\begin{Lem}\label{lem:test2}
For $\ell, T, R \gg 1$, the function $\varphi$ defined by \eqref{test2} belongs to $\Phi$.
\end{Lem}

\begin{proof}
The proof is similar to that of Lemma~\ref{lem:test1} and is therefore omitted.
\end{proof}

\section{Auxiliary estimates}\label{sec3}

In this section, we establish auxiliary estimates used later in the proof of the nonexistence results. We recall that $p > 1$, $\sigma > -4$, $N \ge 5$, and that $\mu$ satisfies \eqref{mu-as}.

In what follows, $C$ denotes a positive constant independent of $T$ and $R$, whose value may change from line to line. We also set
\[
p' = \frac{p}{p-1}.
\]

\subsection{The case $0 < \mu < \mu^*$}

Throughout this subsection, we assume that
\[
0 < \mu < \mu^*,
\]
and that $H$ is given by \eqref{H-f}.

\begin{Lem}\label{L3.1}
For $R \gg 1$, one has
\[
\int_{\Omega_e} H(x)\, \xi_R(x)\, |x|^{-\frac{\sigma}{p-1}} \, dx
\le
C \left( \ln R + R^{N + \alpha_+ - \frac{\sigma}{p-1}} \right).
\]
\end{Lem}

\begin{proof}
From the properties of the cutoff function $\xi$, we have
\[
0 \le \xi_R \le 1, \qquad \supp \xi_R \subset \{ x \in \mathbb{R}^N : 1 \le |x| \le R \}.
\]
Therefore,
\[
\int_{\Omega_e} H(x)\, \xi_R(x)\, |x|^{-\frac{\sigma}{p-1}} \, dx
\le
\int_{1 < |x| < R} H(x)\, |x|^{-\frac{\sigma}{p-1}} \, dx.
\]
On the other hand, since $\beta_- < \alpha_- < \alpha_+$, we have
\[
0 \le H(x) \le C |x|^{\alpha_+}, \qquad 1 < |x| < R.
\]
Hence,
\[
\int_{\Omega_e} H(x)\, \xi_R(x)\, |x|^{-\frac{\sigma}{p-1}} \, dx
\le
C \int_{1 < |x| < R} |x|^{\alpha_+ - \frac{\sigma}{p-1}} \, dx.
\]
Passing to polar coordinates yields
\[
\begin{aligned}
\int_{\Omega_e} H(x)\, \xi_R(x)\, |x|^{-\frac{\sigma}{p-1}} \, dx
&\le
C \int_1^R r^{N-1 + \alpha_+ - \frac{\sigma}{p-1}} \, dr \\
&\le
C \left( \ln R + R^{N + \alpha_+ - \frac{\sigma}{p-1}} \right).
\end{aligned}
\]
The proof is complete.
\end{proof}

\begin{Lem}\label{L3.2}
For $\ell, R \gg 1$, one has
\[
\int_{\Omega_e}
\frac{\left| \mathcal{L}_\mu (H \xi_R)(x) \right|^{p'}}
{\big( H(x)\, \xi_R(x) \big)^{p'-1}}
\, |x|^{-\frac{\sigma}{p-1}} \, dx
\le
C\, R^{N + \alpha_+ - 4p' - \frac{\sigma}{p-1}}.
\]
\end{Lem}

\begin{proof}
For $x \in \Omega_e$, we have
\[
\mathcal{L}_\mu (H \xi_R)
=
\Delta^2 (H \xi_R) - \frac{\mu}{|x|^4} H \xi_R.
\]
Using the product rule for $\Delta^2 (H \xi_R)$, we obtain
\begin{align*}
\Delta^2 (H \xi_R)
&=
\xi_R\, \Delta^2 H
+ H\, \Delta^2 \xi_R
+ 2 (\Delta H)(\Delta \xi_R) \\
&\quad
+ 4 \nabla H \cdot \nabla(\Delta \xi_R)
+ 4 \nabla \xi_R \cdot \nabla(\Delta H)
+ 4 \sum_{i,j=1}^N H_{ij}\, (\xi_R)_{ij}.
\end{align*}
Therefore,
\begin{align*}
\mathcal{L}_\mu (H \xi_R)
&=
\xi_R\, \mathcal{L}_\mu H
+ H\, \Delta^2 \xi_R
+ 2 (\Delta H)(\Delta \xi_R) \\
&\quad
+ 4 \nabla H \cdot \nabla(\Delta \xi_R)
+ 4 \nabla \xi_R \cdot \nabla(\Delta H)
+ 4 \sum_{i,j=1}^N H_{ij}\, (\xi_R)_{ij}.
\end{align*}
Since $\mathcal{L}_\mu H = 0$ (by Lemma~\ref{L2.2}), the first term vanishes. Hence,
\begin{equation}\label{S1-L3.2}
\begin{aligned}
\mathcal{L}_\mu (H \xi_R)
&=
H\, \Delta^2 \xi_R
+ 2 (\Delta H)(\Delta \xi_R) \\
&\quad
+ 4 \nabla H \cdot \nabla(\Delta \xi_R)
+ 4 \nabla \xi_R \cdot \nabla(\Delta H)
+ 4 \sum_{i,j=1}^N H_{ij}\, (\xi_R)_{ij}.
\end{aligned}
\end{equation}
Here $H_{ij} = \partial_{x_i x_j}^2 H$ and $(\xi_R)_{ij} = \partial_{x_i x_j}^2 \xi_R$.

Since $\mathcal{L}_\mu (H \xi_R)$ involves only derivatives of $\xi_R$, and these derivatives are supported in the annulus
\[
\frac{R}{2} < |x| < R,
\]
we obtain
\begin{equation}\label{S2-L3.2}
\int_{\Omega_e}
\frac{\left| \mathcal{L}_\mu (H \xi_R)(x) \right|^{p'}}
{\big( H(x)\, \xi_R(x) \big)^{p'-1}}
\, |x|^{-\frac{\sigma}{p-1}} \, dx
=
\int_{\frac{R}{2} < |x| < R}
\frac{\left| \mathcal{L}_\mu (H \xi_R)(x) \right|^{p'}}
{\big( H(x)\, \xi_R(x) \big)^{p'-1}}
\, |x|^{-\frac{\sigma}{p-1}} \, dx.
\end{equation}

By the definitions of $H$ and $\xi_R$, one easily verifies that, for $\frac{R}{2} < |x| < R$,
\[
\begin{aligned}
H(x)\, |\Delta^2 \xi_R(x)|
&\le
C |x|^{\alpha_+ - 4}\, \xi^{\ell-4}\!\left(\frac{|x|}{R}\right), \quad
|\Delta H(x)|\, |\Delta \xi_R(x)|
\le
C |x|^{\alpha_+ - 4}\, \xi^{\ell-2}\!\left(\frac{|x|}{R}\right), \\
|\nabla H(x)|\, |\nabla(\Delta \xi_R)(x)|
&\le
C |x|^{\alpha_+ - 4}\, \xi^{\ell-3}\!\left(\frac{|x|}{R}\right), \quad
|\nabla \xi_R(x)|\, |\nabla(\Delta H)(x)|
\le
C |x|^{\alpha_+ - 4}\, \xi^{\ell-1}\!\left(\frac{|x|}{R}\right),
\end{aligned}
\]
and
\[
\sum_{i,j=1}^N |H_{ij}(x)|\, |(\xi_R)_{ij}(x)|
\le
C |x|^{\alpha_+ - 4}\, \xi^{\ell-2}\!\left(\frac{|x|}{R}\right).
\]

From \eqref{S1-L3.2} and the fact that $0 \le \xi \le 1$, we deduce that, for $\frac{R}{2} < |x| < R$,
\[
\left| \mathcal{L}_\mu (H \xi_R)(x) \right|
\le
C |x|^{\alpha_+ - 4} \, \xi^{\ell-4}\!\left(\frac{|x|}{R}\right),
\]
which yields (for $\ell \ge 4p'$)
\[
\begin{aligned}
\frac{\left| \mathcal{L}_\mu (H \xi_R)(x) \right|^{p'}}
{\big( H(x)\, \xi_R(x) \big)^{p'-1}}
\, |x|^{-\frac{\sigma}{p-1}}
&\le
C
\frac{|x|^{(\alpha_+ - 4)p' - \frac{\sigma}{p-1}}}
{H(x)^{p'-1}}
\, \xi^{\ell - 4p'}\!\left(\frac{|x|}{R}\right) \\
&\le
C \frac{|x|^{(\alpha_+ - 4)p' - \frac{\sigma}{p-1}}}
{H(x)^{p'-1}}.
\end{aligned}
\]
On the other hand, for $R \gg 1$, we have $H(x) \sim |x|^{\alpha_+}$ for $\frac{R}{2} < |x| < R$. Therefore,
\[
\frac{\left| \mathcal{L}_\mu (H \xi_R)(x) \right|^{p'}}
{\big( H(x)\, \xi_R(x) \big)^{p'-1}}
\, |x|^{-\frac{\sigma}{p-1}}
\le
C |x|^{\alpha_+ - 4p' - \frac{\sigma}{p-1}},
\qquad \frac{R}{2} < |x| < R.
\]
Using \eqref{S2-L3.2}, we obtain
\[
\begin{aligned}
\int_{\Omega_e}
\frac{\left| \mathcal{L}_\mu (H \xi_R)(x) \right|^{p'}}
{\big( H(x)\, \xi_R(x) \big)^{p'-1}}
\, |x|^{-\frac{\sigma}{p-1}} \, dx
&\le
C \int_{\frac{R}{2} < |x| < R} |x|^{\alpha_+ - 4p' - \frac{\sigma}{p-1}} \, dx \\
&\le
C R^{N + \alpha_+ - 4p' - \frac{\sigma}{p-1}}.
\end{aligned}
\]
This completes the proof.
\end{proof}

\begin{Lem}\label{L3.3}
For $R \gg 1$, one has
\[
\int_{\Omega_e} H(x)\, \zeta_R(x)\, |x|^{-\frac{\sigma}{p-1}} \, dx
\le
C \left( \ln R + R^{N + \alpha_+ - \frac{\sigma}{p-1}} \right).
\]
\end{Lem}

\begin{proof}
The proof follows the same arguments as in Lemma~\ref{L3.1} and is therefore omitted.
\end{proof}

\begin{Lem}\label{L3.4}
For $\ell, R \gg 1$, and
\[
p = 1 + \frac{\sigma + 4}{N - 4 + \alpha_+},
\]
one has
\[
\int_{\Omega_e}
\frac{\left| \mathcal{L}_\mu (H \zeta_R)(x) \right|^{p'}}
{\big( H(x)\, \zeta_R(x) \big)^{p'-1}}
\, |x|^{-\frac{\sigma}{p-1}} \, dx
\le
C (\ln R)^{1 - p'}.
\]
\end{Lem}

\begin{proof}
The proof follows the same arguments as in Lemma~\ref{L3.2}. We only indicate the main difference. For the logarithmic cutoff $\zeta_R$, each spatial derivative produces an additional factor $(\ln R)^{-1}$, and the cutoff error is supported in
\[
R^{1/2} < |x| < R.
\]
Moreover, the choice
\[
p = 1 + \frac{\sigma + 4}{N - 4 + \alpha_+}
\]
is equivalent to
\[
N - 1 + \alpha_+ - 4p' - \frac{\sigma}{p-1} = -1.
\]
Thus, arguing as in Lemma~\ref{L3.2}, we obtain
\[
\begin{aligned}
\int_{\Omega_e}
\frac{\left| \mathcal{L}_\mu (H \zeta_R)(x) \right|^{p'}}
{\big( H(x)\, \zeta_R(x) \big)^{p'-1}}
\, |x|^{-\frac{\sigma}{p-1}} \, dx
&\le
C (\ln R)^{-p'} \int_{R^{1/2}}^R r^{-1} \, dr \\
&\le
C (\ln R)^{1 - p'},
\end{aligned}
\]
which proves the estimate.
\end{proof}

\subsection{The case $\mu = \mu^*$}

Throughout this subsection, we assume that
\[
\mu = \mu^*,
\]
and that $H$ is given by \eqref{H-critical}.

\begin{Lem}\label{L3.5}
For $R \gg 1$, one has
\[
\int_{\Omega_e} H(x)\, \xi_R(x)\, |x|^{-\frac{\sigma}{p-1}} \, dx
\le
C\, \ln R \left( \ln R + R^{N + \alpha - \frac{\sigma}{p-1}} \right).
\]
\end{Lem}

\begin{proof}
The proof follows the same arguments as in Lemma~\ref{L3.1}. We only indicate the difference. In the present case, we have
\[
H(x) \sim |x|^\alpha \ln |x|
\qquad \text{for } |x| \gg 1,
\]
which introduces an additional logarithmic factor. The conclusion then follows by the same argument.
\end{proof}

\begin{Lem}\label{L3.6}
For $\ell, R \gg 1$, one has
\[
\int_{\Omega_e}
\frac{\left| \mathcal{L}_\mu (H \xi_R)(x) \right|^{p'}}
{\big( H(x)\, \xi_R(x) \big)^{p'-1}}
\, |x|^{-\frac{\sigma}{p-1}} \, dx
\le
C\, R^{N + \alpha - 4p' - \frac{\sigma}{p-1}} \ln R.
\]
\end{Lem}

\begin{proof}
The proof follows the same lines as in Lemma~\ref{L3.2}, taking into account that, in the case $\mu = \mu^*$, the function $H$ satisfies
\[
H(x) \sim |x|^\alpha \ln |x|
\qquad \text{for } |x| \gg 1,
\]
which introduces an additional logarithmic factor.
\end{proof}

\subsection{Integral estimates involving $\varrho_T$}

\begin{Lem}\label{L3.7}
For $\ell, T \gg 1$, one has
\[
\int_0^\infty \frac{|\varrho_T'(t)|^{p'}}{\varrho_T(t)^{p'-1}} \, dt
\le
C T^{1 - p'}.
\]
\end{Lem}

\begin{proof}
Since $\supp \varrho \subset (0,1)$, we have $\supp \varrho_T \subset (0,T)$. Hence,
\[
\int_0^\infty \frac{|\varrho_T'(t)|^{p'}}{\varrho_T(t)^{p'-1}} \, dt
=
\int_0^T \frac{|\varrho_T'(t)|^{p'}}{\varrho_T(t)^{p'-1}} \, dt.
\]
Moreover,
\[
\varrho_T'(t)
=
\frac{\ell}{T} \, \varrho^{\ell-1}\!\left(\frac{t}{T}\right) \varrho'\!\left(\frac{t}{T}\right).
\]
Thus,
\[
\int_0^T \frac{|\varrho_T'(t)|^{p'}}{\varrho_T(t)^{p'-1}} \, dt
\le
C T^{-p'} \int_0^T
\varrho^{\ell - p'}\!\left(\frac{t}{T}\right)
\left| \varrho'\!\left(\frac{t}{T}\right) \right|^{p'} \, dt.
\]
Using the change of variables $s = t/T$, we obtain (for $\ell \ge p'$)
\[
\begin{aligned}
\int_0^\infty \frac{|\varrho_T'(t)|^{p'}}{\varrho_T(t)^{p'-1}} \, dt
&\le
C T^{1 - p'} \int_0^1 \varrho^{\ell - p'}(s) |\varrho'(s)|^{p'} \, ds \\
&\le
C T^{1 - p'}.
\end{aligned}
\]
The proof is complete.
\end{proof}

\begin{Lem}\label{lem:rhoT}
For $\ell, T \gg 1$, one has
\[
\int_0^\infty \varrho_T(t)\, dt = C T.
\]
\end{Lem}

\begin{proof}
The result follows directly from the definition of $\varrho_T$.
\end{proof}

\section{Proofs of the main results}\label{sec4}

This section is devoted to the proofs of Theorems~\ref{T1.2} and \ref{T1.3}.

\begin{proof}[Proof of Theorem~\ref{T1.2}]
\textbf{Part (i).} Assume that one of the conditions \eqref{cd-nex} or \eqref{cd-nex-cr} holds. By contradiction, suppose that problem \eqref{P}--\eqref{BC} admits a weak solution $u \in L^p_{\mathrm{loc}}(D)$ for some positive function $f \in L^1_{\mathrm{loc}}(\overline{\Omega_e})$.

By Definition~\ref{def-ws}, for all $\varphi \in \Phi$, one has
\begin{equation}\label{step1-T1}
\int_D |x|^{\sigma} |u|^p \varphi \, dx\, dt
+ \int_D f(x)\, \varphi \, dx\, dt
\le
\int_D |u|\, |\varphi_t| \, dx\, dt
+ \int_D |u|\, |\mathcal{L}_\mu \varphi| \, dx\, dt.
\end{equation}
By Young's inequality, one has
\[
\int_D |u|\, |\varphi_t| \, dx\, dt
\le
\frac{1}{2} \int_D |x|^{\sigma} |u|^p \varphi \, dx\, dt
+ C \int_D \frac{|\varphi_t|^{p'}}{\varphi^{p'-1}} \, |x|^{-\frac{\sigma}{p-1}} \, dx\, dt.
\]
Similarly,
\[
\int_D |u|\, |\mathcal{L}_\mu \varphi| \, dx\, dt
\le
\frac{1}{2} \int_D |x|^{\sigma} |u|^p \varphi \, dx\, dt
+ C \int_D \frac{|\mathcal{L}_\mu \varphi|^{p'}}{\varphi^{p'-1}} \, |x|^{-\frac{\sigma}{p-1}} \, dx\, dt.
\]

Substituting the above estimates into \eqref{step1-T1}, we obtain
\begin{equation}\label{ap-est}
\int_D f(x)\, \varphi \, dx\, dt
\le
C \left( \mathcal{C}_1(\varphi) + \mathcal{C}_2(\varphi) \right),
\end{equation}
where
\[
\mathcal{C}_1(\varphi)
=
\int_D \frac{|\varphi_t|^{p'}}{\varphi^{p'-1}} \, |x|^{-\frac{\sigma}{p-1}} \, dx\, dt,
\qquad
\mathcal{C}_2(\varphi)
=
\int_D \frac{|\mathcal{L}_\mu \varphi|^{p'}}{\varphi^{p'-1}} \, |x|^{-\frac{\sigma}{p-1}} \, dx\, dt,
\]
provided that $\mathcal{C}_1(\varphi)$ and $\mathcal{C}_2(\varphi)$ are finite.

We distinguish three cases.

\noindent \textbf{Case 1.} $0 < \mu < \mu^*$ and $1 < p < 1 + \frac{\sigma + 4}{\mu_N}$. In this case, for $\ell, T, R \gg 1$, we consider the function $\varphi$ defined by \eqref{test1}, namely,
\[
\varphi(t,x) = \varrho_T(t)\, H_R(x), \qquad (t,x) \in D,
\]
where
\[
H_R(x) = H(x)\, \xi_R(x), \qquad x \in \overline{\Omega_e},
\]
and $H$ is given by \eqref{H-f}.

We first estimate $\mathcal{C}_1(\varphi)$. By definition of $\varphi$, we obtain
\[
\mathcal{C}_1(\varphi)
=
\mathcal{C}_{11}(T)\, \mathcal{C}_{12}(R),
\]
where
\[
\mathcal{C}_{11}(T)
=
\int_0^\infty
\frac{|\varrho_T'(t)|^{p'}}{\varrho_T(t)^{p'-1}} \, dt
\]
and
\[
\mathcal{C}_{12}(R)
=
\int_{\Omega_e}
H(x)\, \xi_R(x)\, |x|^{-\frac{\sigma}{p-1}} \, dx.
\]
By Lemma~\ref{L3.7}, one has
\[
\mathcal{C}_{11}(T) \le C T^{1 - p'}.
\]
Moreover, by Lemma~\ref{L3.1}, one has
\[
\mathcal{C}_{12}(R)
\le
C \left( \ln R + R^{N + \alpha_+ - \frac{\sigma}{p-1}} \right).
\]
Consequently,
\begin{equation}\label{C1-case1}
\mathcal{C}_1(\varphi)
\le
C T^{1 - p'}
\left( \ln R + R^{N + \alpha_+ - \frac{\sigma}{p-1}} \right).
\end{equation}

We next estimate $\mathcal{C}_2(\varphi)$. By definition of $\varphi$, we obtain
\[
\mathcal{C}_2(\varphi)
=
\mathcal{C}_{21}(T)\, \mathcal{C}_{22}(R),
\]
where
\[
\mathcal{C}_{21}(T)
=
\int_0^\infty \varrho_T(t)\, dt
\]
and
\[
\mathcal{C}_{22}(R)
=
\int_{\Omega_e}
\frac{\left| \mathcal{L}_\mu H_R(x) \right|^{p'}}
{H_R(x)^{p'-1}}
\, |x|^{-\frac{\sigma}{p-1}} \, dx.
\]
By Lemmas~\ref{lem:rhoT} and~\ref{L3.2}, one has
\[
\mathcal{C}_{21}(T) = C T
\]
and
\[
\mathcal{C}_{22}(R)
\le
C R^{N + \alpha_+ - 4p' - \frac{\sigma}{p-1}}.
\]
Therefore,
\begin{equation}\label{C2-case1}
\mathcal{C}_2(\varphi)
\le
C T R^{N + \alpha_+ - 4p' - \frac{\sigma}{p-1}}.
\end{equation}

Combining \eqref{ap-est}, \eqref{C1-case1}, and \eqref{C2-case1}, we obtain
\begin{equation}\label{res-case1}
\int_D f(x)\, \varphi \, dx\, dt
\le
C \left[
T^{1 - p'} \left( \ln R + R^{N + \alpha_+ - \frac{\sigma}{p-1}} \right)
+ T R^{N + \alpha_+ - 4p' - \frac{\sigma}{p-1}}
\right].
\end{equation}

We now estimate the left-hand side of \eqref{res-case1}. By definition of $\varphi$, we have
\[
\int_D f(x)\, \varphi(t,x)\, dx\, dt
=
\left( \int_0^\infty \varrho_T(t)\, dt \right)
\left( \int_{\Omega_e} f(x)\, H(x)\, \xi_R(x)\, dx \right).
\]
By Lemma~\ref{lem:rhoT}, we obtain
\[
\int_D f(x)\, \varphi(t,x)\, dx\, dt
=
C T
\left( \int_{\Omega_e} f(x)\, H(x)\, \xi_R(x)\, dx \right).
\]
Moreover, since $f > 0$, $f \in L^1_{\mathrm{loc}}(\overline{\Omega_e})$, $H \ge 0$ (by Lemma~\ref{L2.2}), and $\xi_R(x) = 1$ for $1 < |x| < R/2$, we have
\[
\int_{\Omega_e} f(x)\, H(x)\, \xi_R(x)\, dx
\ge
\int_{1 < |x| < R/2} f(x)\, H(x)\, dx.
\]
Fix $\tau > 2$ such that $R > 2\tau$. Then
\[
\int_{1 < |x| < R/2} f(x)\, H(x)\, dx
\ge
\int_{2 < |x| < \tau} f(x)\, H(x)\, dx
> 0.
\]
Therefore,
\[
\int_D f(x)\, \varphi(t,x)\, dx\, dt \ge C T.
\]

Then, by \eqref{res-case1}, we deduce that
\[
T
\le
C \left[
T^{1 - p'} \left( \ln R + R^{N + \alpha_+ - \frac{\sigma}{p-1}} \right)
+ T R^{N + \alpha_+ - 4p' - \frac{\sigma}{p-1}}
\right].
\]
Dividing by $T$, we obtain
\[
1
\le
C \left[
T^{-p'} \left( \ln R + R^{N + \alpha_+ - \frac{\sigma}{p-1}} \right)
+ R^{N + \alpha_+ - 4p' - \frac{\sigma}{p-1}}
\right].
\]
On the other hand, by the definition of $\alpha_+$, we have
\[
N + \alpha_+ - 4p' - \frac{\sigma}{p-1}
=
\mu_N - \frac{\sigma + 4}{p-1}.
\]
Therefore, taking $T = R^\theta$, the above estimate becomes
\begin{equation}\label{OK-case1}
1
\le
C \left[
R^{-\theta p'} \left( \ln R + R^{N + \alpha_+ - \frac{\sigma}{p-1}} \right)
+ R^{\mu_N - \frac{\sigma + 4}{p-1}}
\right].
\end{equation}
Since
\[
1 < p < 1 + \frac{\sigma + 4}{\mu_N},
\]
we have
\[
\mu_N - \frac{\sigma + 4}{p-1} < 0.
\]
Moreover, choosing
\[
\theta > \max \left\{ 0, \frac{N + \alpha_+ - \frac{\sigma}{p-1}}{p'} \right\},
\]
we obtain
\[
R^{-\theta p'} \left( \ln R + R^{N + \alpha_+ - \frac{\sigma}{p-1}} \right) \to 0
\qquad \text{as } R \to \infty.
\]
Hence, passing to the limit as $R \to \infty$ in \eqref{OK-case1}, we obtain a contradiction. This shows that problem \eqref{P}--\eqref{BC} admits no weak solution for $0<\mu<\mu^*$ and $1 < p < 1 + \frac{\sigma + 4}{\mu_N}$.

\noindent \textbf{Case 2.} $0 < \mu < \mu^*$ and $p = 1 + \frac{\sigma + 4}{\mu_N}$. In this case, for $\ell, T, R \gg 1$, we consider the function $\varphi$ defined by \eqref{test2}, namely,
\[
\varphi(t,x) = \varrho_T(t)\, \overline{H}_R(x), \qquad (t,x) \in D,
\]
where
\[
\overline{H}_R(x) = H(x)\, \zeta_R(x), \qquad x \in \overline{\Omega_e},
\]
and $H$ is given by \eqref{H-f}.

Proceeding as in the previous case and using Lemmas~\ref{L3.7} and~\ref{L3.3}, we obtain
\begin{equation}\label{C1-case2}
\mathcal{C}_1(\varphi)
\le
C T^{1 - p'}
\left( \ln R + R^{N + \alpha_+ - \frac{\sigma}{p-1}} \right).
\end{equation}
Using Lemmas~\ref{lem:rhoT} and~\ref{L3.4}, and recalling that
\[
N - 4 + \alpha_+ = \mu_N,
\]
we obtain
\begin{equation}\label{C2-case2}
\mathcal{C}_2(\varphi)
\le
C T (\ln R)^{1 - p'}.
\end{equation}

Combining \eqref{ap-est}, \eqref{C1-case2}, and \eqref{C2-case2}, we obtain
\[
\int_D f(x)\, \varphi \, dx\, dt
\le
C \left[
T^{1 - p'} \left( \ln R + R^{N + \alpha_+ - \frac{\sigma}{p-1}} \right)
+ T (\ln R)^{1 - p'}
\right].
\]
Arguing as in the previous case and using the positivity of $f$ and $H$, we obtain
\[
\int_D f(x)\, \varphi(t,x)\, dx\, dt \ge C T.
\]
Hence,
\[
1
\le
C \left[
T^{-p'} \left( \ln R + R^{N + \alpha_+ - \frac{\sigma}{p-1}} \right)
+ (\ln R)^{1 - p'}
\right].
\]
Choosing $T = R^\theta$ with $\theta > 0$ sufficiently large, and passing to the limit as $R \to \infty$, we obtain a contradiction. This proves the nonexistence in the critical case.

\noindent \textbf{Case 3.} $\mu = \mu^*$ and $1<p < 1 + \frac{\sigma + 4}{\mu_N}$. For $\ell, T, R \gg 1$, we consider the function $\varphi$ defined by \eqref{test1}, namely,
\[
\varphi(t,x) = \varrho_T(t)\, H_R(x), \qquad (t,x) \in D,
\]
where
\[
H_R(x) = H(x)\, \xi_R(x), \qquad x \in \overline{\Omega_e},
\]
and $H$ is given by \eqref{H-critical}.

Using Lemmas~\ref{L3.7} and~\ref{L3.5}, we obtain
\[
\mathcal{C}_1(\varphi)
\le
C T^{1 - p'}\, \ln R
\left( \ln R + R^{N + \alpha - \frac{\sigma}{p-1}} \right).
\]
Moreover, by Lemmas~\ref{lem:rhoT} and~\ref{L3.6}, one has
\[
\mathcal{C}_2(\varphi)
\le
C T\, R^{N + \alpha - 4p' - \frac{\sigma}{p-1}}\, \ln R.
\]
Furthermore, arguing as in the previous cases, we obtain
\[
\int_D f(x)\, \varphi(t,x)\, dx\, dt \ge C T.
\]

Combining the above estimates with \eqref{ap-est}, we obtain
\[
1
\le
C \left[
T^{-p'}\, \ln R
\left( \ln R + R^{N + \alpha - \frac{\sigma}{p-1}} \right)
+
R^{N + \alpha - 4p' - \frac{\sigma}{p-1}}\, \ln R
\right].
\]
On the other hand, by the definition of the double root $\alpha$, one has
\[
1 < p < 1 + \frac{\sigma + 4}{\mu_N}
\quad \Longleftrightarrow \quad
N + \alpha - 4p' - \frac{\sigma}{p-1} < 0.
\]
Hence, choosing $T = R^\theta$ with $\theta > 0$ sufficiently large and letting $R \to \infty$ in the above estimate, we obtain a contradiction. This completes the proof in this case.\\

\textbf{Part (ii).} Let $p$ satisfy \eqref{cd-ex}. We treat separately the cases $0 < \mu < \mu^*$ and $\mu = \mu^*$.

\noindent \textbf{Case 1.} $0 < \mu < \mu^*$. Let $q$ be such that
\[
\max\left\{ -\alpha_+, \frac{\sigma + 4}{p-1} \right\} < q < \mu_N.
\]
Such a choice is possible. Indeed, since
\[
p > 1 + \frac{\sigma + 4}{\mu_N},
\]
we have
\[
\frac{\sigma + 4}{p-1} < \mu_N.
\]
Moreover,
\[
-\alpha_+ < -\alpha_- = \mu_N.
\]
Hence,
\[
\max\left\{ -\alpha_+, \frac{\sigma + 4}{p-1} \right\} < \mu_N,
\]
which ensures that such a choice of $q$ exists.

We define
\[
V_q(x) = |x|^{-q} + A_q |x|^{\alpha_-} + B_q |x|^{\beta_-},
\qquad |x| \ge 1,
\]
where
\[
A_q = \frac{M(\beta_-) - M(-q)}{M(\alpha_-) - M(\beta_-)},
\qquad
B_q = \frac{M(-q) - M(\alpha_-)}{M(\alpha_-) - M(\beta_-)},
\]
and $M$ is defined by  \eqref{Mlambda}.

A direct computation shows that
\[
V_q(x) = \Delta V_q(x) = 0, \qquad x \in \partial B_1,
\]
and
\[
\mathcal{L}_\mu V_q(x) = P(-q)\, |x|^{-q-4}, \qquad |x| > 1,
\]
where $P$ is the polynomial defined by \eqref{Plambda}. Moreover, since $\alpha_- < -q < \alpha_+$ and $P$ has simple roots ordered as
\[
\beta_- < \alpha_- < \alpha_+ < 0 < \beta_+,
\]
we have
\[
P(-q) > 0.
\]

Now, we choose $0 < \varepsilon < \varepsilon_0$, where
\begin{equation}\label{eps0}
\varepsilon_0
=
\left(
\frac{P(-q)}
{\displaystyle \sup_{|y| \ge 1} |y|^{\sigma + q + 4}\, |V_q(y)|^p}
\right)^{\frac{1}{p-1}}.
\end{equation}
The above supremum is finite. Indeed, we have
\[
V_q(x) \sim |x|^{-q}
\qquad \text{as } |x| \to \infty.
\]
Therefore,
\[
|x|^{\sigma + q + 4} |V_q(x)|^p
\sim
|x|^{\sigma + q + 4 - pq}
=
|x|^{\sigma + 4 - q(p-1)}.
\]
Since
\[
q > \frac{\sigma + 4}{p-1},
\]
it follows that
\[
\sigma + 4 - q(p-1) < 0,
\]
and hence
\[
|x|^{\sigma + q + 4} |V_q(x)|^p \to 0
\qquad \text{as } |x| \to \infty.
\]
Moreover, the function $x \mapsto |x|^{\sigma + q + 4} |V_q(x)|^p$ is continuous on $\{ |x| \ge 1 \}$. Therefore,
\[
\sup_{|y| \ge 1} |y|^{\sigma + q + 4} |V_q(y)|^p < \infty.
\]

For such $\varepsilon$, we define
\begin{equation}\label{ueps}
u_\varepsilon(x) = \varepsilon V_q(x), \qquad |x| \ge 1,
\end{equation}
and
\[
f_\varepsilon(x) = \mathcal{L}_\mu u_\varepsilon(x) - |x|^{\sigma} |u_\varepsilon(x)|^p, \qquad |x| > 1.
\]
By the definition of $V_q$ and the choice of $\varepsilon$, for all $x$ with $|x| > 1$, we have
\[
\begin{aligned}
f_\varepsilon(x)
&=
\varepsilon |x|^{-q-4}
\left[
P(-q)
- \varepsilon^{p-1} |x|^{\sigma + q + 4} |V_q(x)|^p
\right] \\
&\ge
\varepsilon |x|^{-q-4}
\left[
P(-q)
- \varepsilon^{p-1} \sup_{|y| \ge 1} |y|^{\sigma + q + 4} |V_q(y)|^p
\right] \\
&> 0.
\end{aligned}
\]

We conclude that $u_\varepsilon$ is a stationary solution of \eqref{P}--\eqref{BC} corresponding to the positive source $f_\varepsilon$. Moreover, $u_\varepsilon \in C^\infty(\overline{\Omega_e})$. 

\noindent \textbf{Case 2.} $\mu = \mu^*$. We choose $q$ such that
\[
\frac{\sigma + 4}{p-1} < q < -\alpha.
\]
This choice is possible since
\[
p > 1 + \frac{\sigma + 4}{\mu_N}
\quad \Longleftrightarrow \quad
\frac{\sigma + 4}{p-1} < \mu_N = -\alpha.
\]

We define
\[
V_q(x) = -|x|^{-q} + |x|^\alpha + B_q |x|^\alpha \ln |x|,
\qquad |x| \ge 1,
\]
where
\[
B_q = \frac{M(-q) - M(\alpha)}{2},
\]
and $M$ is defined by \eqref{Mlambda}.

A direct computation shows that
\[
V_q(x) = \Delta V_q(x) = 0, \qquad x \in \partial B_1,
\]
and
\[
\mathcal{L}_\mu V_q(x) = -P(-q)\, |x|^{-q-4}, \qquad |x| > 1,
\]
where $P$ is the polynomial defined by \eqref{Plambda}. Moreover, since $\alpha$ is a double root of $P$ and $\alpha < -q < 0$, we have
\[
P(-q) < 0.
\]

Now, we choose $0 < \varepsilon < \varepsilon_0$, where
\begin{equation}\label{eps0-2}
\varepsilon_0
=
\left(
\frac{-P(-q)}
{\displaystyle \sup_{|y| \ge 1} |y|^{\sigma + q + 4} |V_q(y)|^p}
\right)^{\frac{1}{p-1}}.
\end{equation}
Under the above condition on $q$, the supremum is finite. Indeed, by the definition of $V_q$, we have
\[
|x|^{\sigma + q + 4} |V_q(x)|^p
\sim
|x|^{\sigma + 4 - q(p-1)}
\qquad \text{as } |x| \to \infty.
\]
Since
\[
q > \frac{\sigma + 4}{p-1},
\]
it follows that
\[
|x|^{\sigma + q + 4} |V_q(x)|^p \to 0
\qquad \text{as } |x| \to \infty.
\]

For such $\varepsilon$, we consider the function
\begin{equation}\label{ueps-2}
u_\varepsilon(x) = \varepsilon V_q(x), \qquad |x| \ge 1.
\end{equation}
Next, we define
\[
f_\varepsilon(x) = \mathcal{L}_\mu u_\varepsilon(x) - |x|^{\sigma} |u_\varepsilon(x)|^p, \qquad |x| > 1.
\]
By the definition of $V_q$ and the choice of $\varepsilon$, for all $x$ with $|x| > 1$, we have
\[
\begin{aligned}
f_\varepsilon(x)
&=
\varepsilon |x|^{-q-4}
\left[
- P(-q)
- \varepsilon^{p-1} |x|^{\sigma + q + 4} |V_q(x)|^p
\right] \\
&\ge
\varepsilon |x|^{-q-4}
\left[
- P(-q)
- \varepsilon^{p-1} \sup_{|y| \ge 1} |y|^{\sigma + q + 4} |V_q(y)|^p
\right] \\
&> 0.
\end{aligned}
\]
Consequently, $u_\varepsilon$ is a stationary solution of \eqref{P}--\eqref{BC} corresponding to the positive source $f_\varepsilon$. Moreover, $u_\varepsilon \in C^\infty(\overline{\Omega_e})$. This completes the proof of part (ii).
\end{proof}

\begin{proof}[Proof of Theorem~\ref{T1.3}] \textbf{Part (i).} Let $\eta$ satisfy \eqref{cd-nex-eta}. By contradiction, suppose that problem \eqref{P}--\eqref{BC} admits a weak solution $u \in L^p_{\mathrm{loc}}(D)$ for some $f \in \mathcal{F}^+_{\eta}$.

\noindent \textbf{Case 1.} $0 < \mu < \mu^*$. Following the same steps as in the proof of part (i) of Theorem~\ref{T1.2}, Case 1, we obtain \eqref{res-case1}, where
\[
\int_D f(x)\, \varphi(t,x)\, dx\, dt
=
C T
\left( \int_{\Omega_e} f(x)\, H(x)\, \xi_R(x)\, dx \right), 
\]
and $H$ is given by \eqref{H-f}. 

Since $f \in \mathcal{F}^+_{\eta}$, there exist $c > 0$ and $R_0 > 1$ such that
\[
f(x) \ge c |x|^{-\eta}
\qquad \text{for } |x| \ge R_0.
\]
For $R \gg 1$, using that $\xi_R = 1$ on $\{ R/4 < |x| < R/2 \}$ and that $H(x) \sim |x|^{\alpha_+}$, we obtain
\[
\int_{\Omega_e} f(x)\, H(x)\, \xi_R(x)\, dx
\ge
C \int_{R/4 < |x| < R/2} |x|^{\alpha_+ - \eta} \, dx
\ge
C R^{N + \alpha_+ - \eta}.
\]
Hence,
\[
\int_D f(x)\, \varphi(t,x)\, dx\, dt
\ge
C T R^{N + \alpha_+ - \eta}.
\]
Substituting this estimate into \eqref{res-case1} and dividing by $T R^{N + \alpha_+ - \eta}$, we obtain
\[
1
\le
C \left[
T^{-p'} R^{-N - \alpha_+ + \eta}
\left( \ln R + R^{N + \alpha_+ - \frac{\sigma}{p-1}} \right)
+ R^{\eta - 4p' - \frac{\sigma}{p-1}}
\right].
\]
Since $\eta$ satisfies \eqref{cd-nex-eta}, we have
\[
\eta - 4p' - \frac{\sigma}{p-1} < 0.
\]
Thus, taking $T = R^\theta$, with $\theta > 0$ sufficiently large, and letting $R \to \infty$, we obtain a contradiction.

\noindent \textbf{Case 2.} $\mu = \mu^*$. Following the same steps as in the proof of part (i) of Theorem~\ref{T1.2}, Case 3,  we obtain
\[
\int_{\Omega_e} f(x)\, H(x)\, \xi_R(x)\, dx
\le
C \left[
T^{- p'}\, \ln R
\left( \ln R + R^{N + \alpha - \frac{\sigma}{p-1}} \right)
+ R^{N + \alpha - 4p' - \frac{\sigma}{p-1}}\, \ln R
\right],
\]
where $H$ is given by \eqref{H-critical}. 

Using
\[
H(x) \sim |x|^\alpha \ln |x|
\qquad \text{as } |x| \to \infty,
\]
the same argument as above yields
\[
\int_{\Omega_e} f(x)\, H(x)\, \xi_R(x)\, dx
\ge
C R^{N + \alpha - \eta} \ln R,
\]
for $R \gg 1$.

Combining the above estimates, we obtain
\[
1
\le
C \left[
T^{-p'} R^{-N - \alpha + \eta}
\left( \ln R + R^{N + \alpha - \frac{\sigma}{p-1}} \right)
+ R^{\eta - 4p' - \frac{\sigma}{p-1}}
\right].
\]
The proof then proceeds exactly as in Case 1.

\textbf{Part (ii).}  Assume now that $\eta$ satisfies \eqref{cd-ex-eta}. As before, we consider two cases.

\noindent \textbf{Case 1.} $0 < \mu < \mu^*$. We choose $q$ such that
\[
\max\left\{ -\alpha_+, \frac{\sigma + 4}{p-1}, \eta - 4 \right\} < q < \mu_N.
\]
Such a choice is possible since
\[
p > 1 + \frac{\sigma + 4}{\mu_N}, \qquad
-\alpha_+ < -\alpha_- = \mu_N, \qquad
\eta < \mu_N + 4.
\]

For such $q$ and $0 < \varepsilon < \varepsilon_0$ (see \eqref{eps0}), we define $u_\varepsilon$ by \eqref{ueps} and set
\[
f_\varepsilon(x)
=
\mathcal{L}_\mu u_\varepsilon(x) - |x|^{\sigma} |u_\varepsilon(x)|^p,
\qquad x \in \overline{\Omega_e}.
\]
From the proof of Theorem~\ref{T1.2}, part (ii), Case 1, we have
\[
f_\varepsilon(x)
=
\varepsilon |x|^{-q-4}
\left[
P(-q)
- \varepsilon^{p-1} |x|^{\sigma + q + 4} |V_q(x)|^p
\right]
> 0,
\]
and $f_\varepsilon \in C(\overline{\Omega_e})$.

Moreover, since $P(-q) > 0$ and $-q - 4 < -\eta$, we obtain
\[
f_\varepsilon(x) \le \varepsilon P(-q)\, |x|^{-\eta}, \qquad |x| \ge 1,
\]
so that $f_\varepsilon \in \mathcal{F}^-_{\eta}$. Hence $u_\varepsilon$ is a stationary solution of \eqref{P}--\eqref{BC} with source $f_\varepsilon \in \mathcal{F}^-_{\eta}$, and $u_\varepsilon \in C^\infty(\overline{\Omega_e})$.

\medskip

\noindent \textbf{Case 2.} $\mu = \mu^*$. We choose $q$ such that
\[
\max\left\{ \frac{\sigma + 4}{p-1},\, \eta - 4 \right\} < q < \mu_N.
\]

Proceeding as in Case 1, with $\varepsilon_0$ given by \eqref{eps0-2} and $u_\varepsilon$ defined by \eqref{ueps-2}, we obtain
\[
f_\varepsilon(x)
=
\varepsilon |x|^{-q-4}
\left[
- P(-q)
- \varepsilon^{p-1} |x|^{\sigma + q + 4} |V_q(x)|^p
\right]
> 0,
\]
for all $x \in \overline{\Omega_e}$, and $f_\varepsilon \in C(\overline{\Omega_e})$.

Since $P(-q) < 0$ and $-q - 4 < -\eta$, it follows that
\[
f_\varepsilon(x) \le \varepsilon (-P(-q))\, |x|^{-\eta}, \qquad |x| \ge 1,
\]
so that $f_\varepsilon \in \mathcal{F}^-_{\eta}$. 

This completes the proof of part (ii).

\end{proof}

\medskip 

\noindent{\bf Declaration of competing interest.}
The Authors declare that there is no conflict of interest.

\medskip

\noindent{\bf Funding.}
The third author is supported by Ongoing Research Funding Program, 
(ORF-2026-4), King Saud University, Riyadh, Saudi Arabia.

\end{document}